\documentclass[11pt]{article}

\usepackage[mathscr]{eucal}
\usepackage{color}
\usepackage{amsmath}
\usepackage{amssymb}
\usepackage{amsfonts}
\usepackage{amscd}
\usepackage{amsthm}
\usepackage{latexsym}
\usepackage{graphicx}
\usepackage{subfig}

\def\be{\begin{equation}}
\def\ee{\end{equation}}
\def\bse{\begin{subequations}}
\def\ese{\end{subequations}}

\usepackage{latexsym}

\let\er\eqref

\let\be\beta

\newcommand{\R}{{\mathbb R}}

\newtheorem{theorem}{Theorem}
\newtheorem{lemma}[theorem]{Lemma}

\newtheorem{remark}[theorem]{Remark}

\def\bse{\begin{subequations}}
\def\ese{\end{subequations}}

\usepackage{geometry}

\title{On selection dynamics for a nonlocal phenotype-structured model}

\author{ {\bf\large Shen Bian}\thanks{Corresponding author. Email address: bianshen66@163.com}  \qquad  {\bf\large Jiale Bu}\thanks{ Email address: 2022201105@buct.edu.cn}  \\
\
\\
{\it\small Department of Mathematical Sciences, Beijing University of Chemical Technology} \\
{\it\small Beijing, 100029, P.R. CHINA }\\
\vspace{1.0mm} }
\date{}

\begin{document}
\let\cleardoublepage\clearpage

\maketitle

\begin{abstract}
This paper is devoted to the analysis of the long-time behavior of a phenotypic-structured model where phenotypic changes do not occur. We give a mathematical description of the process in which the best adapted trait is selected in a given environment created by the total population. It is exhibited that the long-time limit of the unique solution to the nonlocal equation is given by a Dirac mass centered at the peak of the fitness within or at the boundary of the region where the initial data is positive. Specially, If the peak of the fitness can't be in the support of the solution, then the infinite time blow-up of the solution occurs near the boundary of the region where the solution is positive. Moreover, our numerical results facilitate a deeper understanding of identifying the position of the centers.
\end{abstract}

\section{Introduction}
In the theory of adaptive evolution (\cite{BMP09,CLL15,D04, DJM05,IM18,JR11,LLH13,LVL18,LLC15,P06,PB08}), a population is structured by a physiological parameter (it is named a trait \cite{P06}), which is hereinafter denoted by . This parameter can represent the size of an organ of the individuals, a proportion of resources utilized for growth and multiplication, or any relevant phenotypic parameter that is helpful to describe the adaptation of the individuals, that is, their ability to utilize the nutrients for reproduction (\cite{CR11,D04,P15}). The main elements in this theory are (i) the selection principle which favors the population with the most well-adapted trait, and (ii) mutations that enable off-springs to have slightly different traits from their mother. These two effects are intensively studied by adaptive dynamics (\cite{BB96,GMK97,JS23,LMS19,LP20,P06,PB08}). A general topic of selection and mutation can also be found in \cite{B00} (in particular, population geneticists might prefer the assumption that mutations are rare rather than small).

In this paper, we mainly focus on an extension of a selection principle. We use $u(x,t)$ to represent the density of individuals with the trait $x \in \R$, and $\rho(t)$ to signify an environment shared by the entire population. We assume that the reproduction rate depends on both the trait and the total population, where the competition among them leads to a decrease in the birth rate, that is, $\frac{b(x)}{1+\rho(t)}$. Moreover, the death rate also depends on the trait and is proportional to the total number of the population, namely
\begin{align}\label{selection11}
\left\{
      \begin{array}{ll}
      \frac{\partial}{\partial t} u(x,t)=\left( \frac{b(x)}{1+c_0 \rho(t)}-d(x) \rho(t) \right) u(x,t)= R(x,\rho(t))u(x,t), ~~& x\in \R,~t>0, \\
      \rho(t)=\int_{\R} u(x,t)dx, ~~& x \in \R,~t>0, \\
      u(x,0)=u_{0}(x) \ge 0, ~~& x \in \R.
      \end{array}\right.
\end{align}
The term $\frac{b(x)}{1+ c_0 \rho(t)}-d(x) \rho(t)$ can be interpreted as the fitness of individuals with the trait $x$ being given the environment created by the total population.

Within the framework of \eqref{selection11}, a mathematical description of phenotypic adaptation can be achieved by examining the long-time behavior of the population density. In this respect, the case where $c_0=0$ (the fitness is linear with respect to $\rho(t)$) has been extensively studied (\cite{ACC22,JS23,LP20,LPS21,P15}). In \cite{LP20}, Lorenzi and Pouchol considered the case of $R(x,\rho(t))=b(x)-\rho(t)$ where $b(x)$ is the net per capita growth rate of the individuals in the phenotypic state $x$ and the saturating term $-\rho(t)$ models the limitations on population growth imposed by carrying capacity constraints. By utilizing the linearity of $R(x,\rho(t))$ in $\rho(t)$ and the semi-explicit formula of $u(x,t)$, they demonstrated the long-time asymptotic behavior of the solution to a selection principle. Later, similar results were obtained in \cite{LPS21,ACC22,JS23} where $R(x,\rho(t))$ is equipped with the form of $b(x)-d(x) \rho(t)$. Comparing to the case of $c_0=0$, we assume $c_0=1$ in our work and introduce a selection principle in a broader context. We shall mainly address the difficulty brought by the nonlinearity of the fitness function with respect to $\rho(t)$ and thereby the long-time limit of the solution is obtained. Specifically, we show that
\begin{itemize}
\item[(i)] The problem \er{selection11} admits a unique global solution $u(x,t)$ and keeps its support conserved at any time.
\item[(ii)] The long-time limit of the solution converges to a Dirac mass concentrated on the set of $x$'s such that the peak of the fitness is attained. (see Theorem \ref{th1}).
\item[(iii)] If the peak of the fitness can't be attained in the support of the solution, then the infinite time blow-up of the solution occurs near the boundary of the support. (see Section \ref{Numerical}).
\end{itemize}
Resorting to the structure of \er{selection11}, we are able to find the semi-explicit formula the solution. Therefore, the result (i) above can be obtained. When proving the asymptotic behavior stated in (ii), we develop a Lyapunov functional to deal with the challenge of the nonlinearity of the fitness with respect to $\rho(t).$ By exploiting the decreasing of the fitness with respect to $\rho(t)$, we prove our desired results via the monotonicity of the Lyapunov functional with respect to time.

The structure of this paper is organized as follows. We commence with a physiologically structured population where a selection principle can be proved (see Theorem \ref{th1}) in Section \ref{Longtime}. Then a sample of numerical solutions that verify the theoretical analysis is constructed in Section \ref{Numerical}. Section \ref{Conclude} concludes the paper by providing a brief overview of possible research perspectives.


\section{Asymptotic analysis}

In this section, we study the asymptotic behaviour of the solution to the Cauchy problem \er{selection11}, which is stated in Theorem \ref{th1}. In order to prove the main result, we first compile the preparatory lemma which has been proved in \cite{P06}.

\begin{lemma}[\cite{P06}]\label{lem:lemma1}
    let $p\in C^1(\R_+)$ and it satisfies
    $\int_{0}^{\infty} \left|\frac{dp(t)}{dt}\right| dt <\infty,$
    then $p(t)$ admits a limit $L$ as $t\to\infty$.
\end{lemma}

To demonstrate the selection principle, we firstly define
\begin{align}
\Omega:=\{x \in \R ~\big |~ u_0(x)>0 \}
\end{align}
and
\begin{align}\label{GG}
G(x,\rho):=\frac{b(x)}{1+\rho(t)}-d(x)\rho(t)
\end{align}
which will be used throughout this section. We assume that $c_0=1$, $b(x), d(x) \in C(\R)$ and there are $b_m, b_M, d_m$ and $d_M$ such that
\begin{align}\label{assume1}
0<b_m<b(x)<b_M,\quad 0<d_m<d(x)<d_M,\quad x \in \R.
\end{align}
We suppose additionally that
\begin{align}\label{assume3}
\frac{b(\bar{x})}{d(\bar{x})}=\max_{x\in \Omega} \frac{b(x)}{d(x)}~~\text{is attained for a single} \bar{x} \in \Omega
\end{align}
and there exists $\bar{\rho}$ such that
\begin{align}\label{assume2}
\frac{b(\bar{x})}{d(\bar{x})}=\bar{\rho}(1+\bar{\rho}).
\end{align}
If $\Omega$ is unbounded, it is still necessary to impose that there exists $R>0$ such that for $R$ large enough and $\rho \to \overline{\rho}$,
\begin{align}\label{assume5}
\alpha_R:=\max_{|x| \ge R} \left[\frac{b(x)}{1+\rho}-d(x)\rho \right]<0.
\end{align}

Under assumptions \er{assume1}-\er{assume5}, there exists a unique non-negative solution $u(x,t) \in C\left( \R_+;L^1(\R) \right)$ of the Cauchy problem \er{selection11} \cite{DJM08,P06}. Besides, the support of the solution remains unchanged as time progresses because of the semi-explicit formula
\begin{align}\label{star}
u(x,t)=u_0(x)~e^{\int_0^t \frac{b(x)}{1+\rho(s)}-d(x) \rho(s) ds  }.
\end{align}
In addition, the best adapted trait (highest reproduction rate) will be selected. Namely, the solution concentrates on the set of $x$'s such that the peak of the fitness is attained when time goes to infinity, as established by the following theorem.

\begin{theorem}[Asymptotic behaviour]\label{th1}
Under assumptions \er{assume1}-\er{assume5}, we assume that there exist $0<\bar{\rho}_m<\bar{\rho}_M<\infty$ such that
\begin{align}\label{assume0}
\bar{\rho}_m \le \rho(0) \le \bar{\rho}_M,
\end{align}
Then the solution to \er{selection11} satisfies
\begin{align}\label{0519}
\rho_m \le \rho(t) \le \rho_M,\quad \forall~ t\ge 0
\end{align}
where
\begin{align}
\rho_m=\min \left( r_m, \rho(0) \right),~~ (1+r_m)r_m=\frac{b_m}{d_M}  ,\\
\rho_M=\max \left( r_M,\rho(0)  \right),~~ (1+r_M)r_M=\frac{b_M}{d_m}.
\end{align}
Furthermore, we have
\begin{align}\label{05191}
\rho(t) \to \bar\rho,~~~ u(x,t) \rightharpoonup \overline{\rho}\delta(x-\bar{x}),~~ \text{as}~~ t \to \infty.
\end{align}
\end{theorem}
\begin{proof}
The proof can be divided into 5 steps. In Step 1, we first give a priori estimate for $\rho(t)$ which plays very important role in proving the long-time limit of $\rho(t)$. Subsequently, a Lyapunov functional is constructed in Steps 2-3 to deduce that there exists a limit for $\rho(t)$ as time goes to infinity. Then, Steps 4-5 identify the limit for $\rho(t)$ and the weak limit for $u(x,t)$.

{\it\textbf{Step 1}} (A priori estimate for $\rho(t)$) \quad In order to prove \er{0519}, we first integrate the equation \er{selection11} in $x$ to obtain
\begin{align}\label{eq3}
    \frac{d}{dt}\rho(t)=\int_{\R} \left(\frac{b(x)}{1+\rho(t)}-d(x)\rho(t)\right) u(x,t) dx.
\end{align}
Using assumption \er{assume1} we further estimate
\begin{align}\label{eq4}
     \frac{d \rho}{dt} & \le \int_{\R} \left( \frac{b_M}{1+\rho(t)}-d_m \rho(t)  \right) u(x,t) dx \nonumber \\
      & = \left( \frac{b_M}{d_m}-\rho (1+\rho)  \right) \frac{d_m \rho}{1+\rho}
\end{align}
and
\begin{align}\label{eq44}
\frac{d \rho}{dt} & \ge \int_{\R} \left( \frac{b_m}{1+\rho(t)}-d_M \rho(t)  \right) u(x,t) dx \nonumber \\
      & = \left( \frac{b_m}{d_M}-\rho (1+\rho) \right) \frac{d_M \rho}{1+\rho}.
\end{align}
Therefore we find that for all times
\begin{align}\label{eq5}
    \min \left( r_m, \rho(0) \right) \le \rho(t) \le \max \left( r_M,\rho(0)  \right)
\end{align}
where $r_m(1+r_m)=\frac{b_m}{d_M}$ and $r_M(1+r_M)=\frac{b_M}{d_m}$.

{\it\textbf{Step 2}} (A Lyapunov functional) \quad Based on \er{eq5}, a Lyapunov functional is constructed to derive the existence of the limit for $\rho(t).$ We first introduce a function
\[
P(\rho)=\frac{1}{3}\rho^2+\frac{1}{2}\rho
\]
satisfying
    \begin{align}\label{eq7}
    \rho P(\rho)'+P(\rho)=Q(\rho):=\rho^2+\rho
    \end{align}
    on the interval $[\rho_m, \rho_M]$. Then we compute
    \begin{align}\label{eq8}
    & \frac{d}{dt} \int_{\R} \left[\frac{b(x)}{d(x)}-P(\rho(t))\right]u(x,t) dx \nonumber \\
    =& \int_{\R} \left( \frac{b(x)}{d(x)}-P(\rho(t))-\rho(t)P'(\rho(t))    \right) \left(\frac{b(x)}{1+\rho(t)}-d(x)\rho(t)\right)u(x,t) dx \nonumber \\
    = & \int_{\R} \frac{1+\rho(t)}{d(x)} \left[\frac{b(x)}{1+\rho(t)}-d(x)\rho(t) \right]^2 u(x, t) dx \ge 0.
    \end{align}
Consequently, $\int_{\R} \left[\frac{b(x)}{d(x)}-P(\rho(t))\right]u(x,t) dx$ is bounded and non-decreasing in time, and thus converges as $t \to \infty $. Namely,
    \begin{align}
    \int_{\R} \left[\frac{b(x)}{d(x)}-P(\rho(t)) \right]u(x,t) dx \to L \in \R,  \label{eq11}\\
    \int_{0}^{\infty} \int_{\R}\left[\frac{b(x)}{1+\rho(t)}-d(x)\rho(t)\right]^2 u(x,t)dxdt <\infty.\label{eq00}
    \end{align}

{\it\textbf{Step 3}} (Existence of the limit for $\rho(t)$) \quad This part focuses on establishing the convergence of $\rho(t)$ as time goes infinity. To this end, recalling the function $ Q(\rho)=\rho^2+\rho$ in \er{eq7} and we do the calculations
\begin{align}\label{eq12}
    &\frac{d}{dt}\int_{\R} \left[\frac{b(x)}{d(x)}-Q(\rho(t))\right]^2 u(x,t)dx \nonumber \\
  = & \int_{\R} \left[\frac{b(x)}{d(x)}-\rho(t)\left(1+\rho(t)\right)\right]^2 \left[\frac{b(x)}{1+\rho(t)}-d(x)\rho(t)\right] u(x,t)dx \nonumber \\
  & -2\left( 2\rho(t)+1 \right) \int_{\R} \left[\frac{b(x)}{d(x)}-\rho(t)\left(1+\rho(t)\right) \right]u(x,t)dx \cdot \int_{\R} \left[ \frac{b(x)}{1+\rho(t)}-d(x)\rho(t) \right]u(x,t)dx \nonumber \\
  =& I_1+I_2.
    \end{align}
Define
\[
   I:=\int_{\R}\left[\frac{b(x)}{1+\rho(t)}-d(x)\rho(t)\right]^2 u(x,t)dx,
\]
the use of the boundedness of $b(x),d(x)$ and $\rho(t)$ in \er{assume1} and \er{eq5} results in
    \begin{align}\label{11}
      |I_1| \le C_1 I,
    \end{align}
    where $C_1$ is a uniformly bounded constant depending on $b_m,b_M,d_m,d_M$ and $\rho_m,\rho_M.$ Moreover, applying Cauchy-Schwarz inequality for $I_2$ we deduce that
    \begin{align}
     |I_2|= & \left|2\left( 2\rho(t)+1 \right) \int_{\R} \frac{1+\rho(t)}{d(x)} [\frac{b(x)}{1+\rho(t)}-d(x) \rho(t) ] u(x,t)dx \cdot \int_{\R} [ \frac{b(x)}{1+\rho(t)}-d(x)\rho(t) ]u(x,t)dx \right| \nonumber \\
    \le & C_2 \left( \int_{\R} [ \frac{b(x)}{1+\rho(t)}-d(x)\rho(t)] u(x,t) dx \right)^2 \nonumber \\
    = & C_2 \left( \int_{\R} [ \frac{b(x)}{1+\rho(t)}-d(x)\rho(t) ] \sqrt{u(x,t)} \sqrt{u(x,t)} dx \right)^2 \nonumber \\
    \le & C_2 \rho(t)  \int_{\R}[\frac{b(x)}{1+\rho(t)}-d(x)\rho(t)]^2 u(x,t)dx, \label{22}
    \end{align}
   where $C_2$ is a constant depending on $\rho_m, \rho_M, d_m,d_M$. Taking \er{11} and \er{22} into account we conclude
    \begin{align}
  \left|\frac{d}{dt}\int_{\R} [\frac{b(x)}{d(x)}-Q(\rho(t)) ]^2 u(x,t)dx \right| \le C I.
   \end{align}
    So \er{eq00} which states that $\int_0^\infty I dt<\infty$ enables us to prove that
   \begin{align}\label{eq22}
    \frac{d}{dt}\int_{\R} [\frac{b(x)}{d(x)}-Q(\rho(t)) ]^2 u(x,t)dx \in L^1(\R_{+}).
    \end{align}
  Therefore, utilizing Lemma \ref{lem:lemma1} gives rise to
  \begin{align}
    \int_{\R} \left[\frac{b(x)}{d(x)}-Q(\rho(t)) \right]^2 u(x,t)dx \to 0,\quad \mbox{as}~~t \to \infty
  \end{align}
    which also arrives at
    \begin{align}\label{eq23}
    \int_{\R}\left| \frac{b(x)}{d(x)}-Q(\rho(t)) \right| u(x, t)dx \to 0, \quad \mbox{as}~~ t \to \infty
    \end{align}
    as a result of Cauchy-Schwarz inequality. Finally, combined with \er{eq11} we have
    \begin{align}\label{eq24}
    \int_{\R} {[\frac{b(x)}{d(x)}-P(\rho(t))]u(x, t)dx}=\int_{\R}{[\frac{b(x)}{d(x)}-Q(\rho(t))+[Q-P](\rho(t))]u(x,t)dx} \to L,
    \end{align}
   which suggests that $\rho(t)~[Q-P](\rho(t))$ converges to a limit and we proceed to obtain
    \begin{align}\label{eq26}
    \rho(t) \to \rho^{*},\quad \mbox{as}~~t \to \infty.
    \end{align}
    Actually, the function $\rho~[Q-P](\rho)$ is not constant locally because
    \begin{align}
    \rho (Q-P)'+Q-P=\rho Q'=\rho(1+2\rho)>0
    \end{align}
    where we have used \er{eq5} and \er{eq7}. In addition, $\rho(t)$ is Lipschitz continuity owing to \er{eq4}, \er{eq44} and \er{eq5}.

{\it\textbf{Step 4}} (Identification of the limit for $\rho(t)$) \quad In this step, we identify that the limit for $\rho(t)$ is $\rho^*=\overline{\rho}$ where $\overline{\rho}$ fulfills \er{assume2}. The proof is revealed by contradiction.
If $\rho^*>\overline{\rho}$, then there exists $t_0$ such that for $t>t_0$ large enough, from \er{assume2} one has
\begin{align}\label{626}
\max_{x \in \R} G(x,\rho)<\max_{x \in \R} G(x,\overline{\rho})=0
\end{align}
since $G(x,\rho)$ is decreasing with respect to $\rho$. Thus the inequality \er{eq4} implies that there is $T>t_0$ such that $\rho(T)=0$ which is impossible because $\rho(t) \ge \rho_m.$ Similarly, if $\rho^*<\overline{\rho}$, then for time large enough we have
\begin{align}
\max_{x \in \R} G(x,\rho)>0.
\end{align}
So $\rho(t)$ admits exponentially growth on the set where $G(x,\rho)>0$ in $x$, which is contradiction due to $\rho(t) \le \rho_M.$  Collecting the two cases together we assert
\begin{align}\label{240520}
\rho(t) \to \rho^{*}=\overline{\rho},~~\mbox{as}~~t \to \infty.
\end{align}

{\it\textbf{Step 5}} (The weak limit for $u(x,t)$) \quad Thanks to \er{star}, we are able to deduce that the support of the solution remains unchanged which reads as
\begin{align}
\{x \in \R ~\big |~ u(x,t)>0 \}=\{x \in \R ~\big |~ u_0(x)>0 \}=\Omega,~~\forall ~t \ge 0.
\end{align}
If $\Omega$ is bounded, it is directly followed that $u(x,t)$ converges weakly to a measure $u^\ast(x)$ equipped with  $\overline{\rho}=\int_{\R} u^\ast(x) dx$. If $\Omega$ is unbounded, \er{assume5} ensures that for $t$ large enough
\begin{align}
\frac{d}{dt} \int_{|x|>R} u(x,t)dx \le \alpha_R \int_{|x|>R} u(x,t) dx
\end{align}
and thereby we have $\sup_{t>0}\int_{|x|>R} u(x,t)dx \to 0$ for $R$ large enough. This indicates that the family $(u(x,t))_{t>0}$ is compact in the weak sense of measures. Thereby there are subsequences $u(x,t_k)$ that converges weakly to to measures $u^\ast(x)$ and $\overline{\rho}=\int_{\R} u^\ast(x) dx$ as $k \to \infty.$

{\it\textbf{Step 6}} (Identification of $u^\ast(x)$) \quad Now we are ready to identify the limit $u^\ast(x)$. It follows from \er{eq23} and \er{240520} that $u^\ast(x)$ concentrates on the set of $x'$s such that $G(x,\overline{\rho})=0.$ Furthermore, with the help of \er{assume3} and \er{assume2}, the point is unique and thus
\begin{align}
u^\ast(x)=\overline{\rho} \delta(x=\bar{x}).
\end{align}
Therefore the family $u(x,t)$ converges uniformly. Indeed, for $x \neq \bar{x},$ we have $\frac{b(x)}{d(x)}<\frac{b(\bar{x})}{d(\bar{x})}$ and thus
\[
u(x,t)=u_0(x)~e^{ -\int_0^t \left(d(x) \rho(s)-\frac{b(x)}{1+\rho(s)} \right) ds }\to 0,~~\text{as}~~ t \to \infty.
\]
Hence the results are proved.
\end{proof}

\begin{remark}\label{th2}
If the unique maximum point of the function $\frac{b(x)}{d(x)}$ in $\Omega$ (we denote it as $\bar{x}$) belongs to the boundary $\partial \Omega$, the the solution to \er{selection11} converges to a half delta function with a mass of $\bar{rho}$. 
\end{remark}

\end{document}